\documentclass[12pt]{amsart}
\usepackage{amsmath,amssymb,amsthm,amsfonts}
\usepackage{latexsym}
\usepackage{graphicx,color}
\def\Ree{\mbox{Re}}

\def\pd#1#2{\frac{\partial #1}{\partial #2}}
\def\vv<#1>{\langle#1\rangle}
\def\w{\wedge}

\def\XXint#1#2{\setbox0=\hbox{$#1{#2}{\int}$}{#2}\kern-.5\wd0 }

\def\XXint#1#2#3{{\setbox0=\hbox{$#1{#2#3}{\int}$}
     \vcenter{\hbox{$#2#3$}}\kern-.5\wd0}}

\def\vv<#1>{\langle#1\rangle}

\newtheorem{theorem}{Theorem}[section]

\newtheorem{lemma}{Lemma}[section]
\newtheorem{proposition}{Proposition}[section]
\newtheorem{corollary}{Corollary}[section]
\theoremstyle{definition}

\theoremstyle{remark}

\newtheorem{remark}{Remark}[section]

\numberwithin{equation}{section}
\title{A Note on K\"ahler-Ricci flow}
\begin{document}

\author{ Chengjie Yu}\address{Department of Mathematics, Shantou University, Shantou, Guangdong, P.R.China
} \email{cjyu@stu.edu.cn}\maketitle
\begin{abstract}
Let $g(t)$ with $t\in [0,T)$ be a complete solution to the
Kaehler-Ricci flow: $\frac{d}{dt}g_{i\bar j}=-R_{i\bar j}$ where $T$
may be $\infty$. In this article, we show that the curvatures of
$g(t)$ is uniformly bounded if the solution $g(t)$ is uniformly
equivalet. This result is stronger than the main result in \v Se\v
sum \cite{sesum} within the category of K\"ahler-Ricci flow.
\end{abstract}

\section{Introduction}

Let $M^n$ be a K\"ahler manifold of complex dimension $n$. Let
$g(t)$ with $t\in [0,T)$ be a complete solution to the
K\"ahler-Ricci flow:
\begin{equation}\label{eqn-Kaehler-Ricci-flow}
\pd{g_{i\bar j}}{t}=-R_{i\bar j}+cg_{i\bar j},
\end{equation}
where $T$ may be $\infty$ and $c$ is a real number. We assume that
$g(t)$ satisfies Shi's estimate (Ref. Shi \cite{Shi-RFU}):
\begin{equation}\label{eqn-Shi-estimate}
\|\nabla^k Rm\|^2(x,t)\leq \frac{C(k,s)}{t^k}
\end{equation}
for any nonnegative integer $k$, any $s\in(0,T)$ and any $(x,t)\in
M\times (0,s]$, where $C(k,s)$ is a positive constant depending on
$k$ and $s$. When $M^n$ is a compact K\"ahler manifold, this
assumption is superfluous because a solution to the K\"ahler-Ricci
flow (\ref{eqn-Kaehler-Ricci-flow}) will automatically satisfy Shi's
estimate (\ref{eqn-Shi-estimate}). By the uniqueness result of
Chen-Zhu \cite{Chen-Zhu-uniqueness}, such a solution to
K\"ahler-Ricci flow (\ref{eqn-Kaehler-Ricci-flow}) is uniquely
determined by its initial metric $g(0)$.

In this article, we obtain the following main result.
\begin{theorem} Let $g(t)$ with $t\in [0,T)$ be a complete solution to the
K\"ahler-Ricci flow (\ref{eqn-Kaehler-Ricci-flow}) satisfying Shi's
estimate (\ref{eqn-Shi-estimate}). Suppose that there is a positive
constant $C$ such that
\begin{equation*}
C^{-1}g(0)\leq g(t)\leq Cg(0).
\end{equation*}
Then, for any nonnegative integer $k$, there are two positive
constants $A_k$ and $B_k$ such that
\begin{equation*}
\|\nabla^kRm\|^2(x,t)\leq A_k+\frac{B_k}{t^k}
\end{equation*}
for any $(x,t)\in M\times [0,T)$.
\end{theorem}

In particular, when $k=0$, this result is stronger than the main
result in \v Se\v sum \cite{sesum} within the category of
K\"ahler-Ricci flow.

This result is useful for obtaining long time existence for
K\"ahler-Ricci flow. As an application, we will give a simple long
time existence of K\"aler-Ricci flow which implies Cao's result (Ref. \cite{Cao}) on
long time existence of K\"ahler-Ricci flow on compact K\"ahler
manifolds.

\vspace{0.5cm}
\noindent\textbf{Acknowledgements.} The author would like to thank Prof. Jianguo Cao, Prof. Peng Lu and
Prof. Xingwang Xu for interesting and encouragement.

\section{A maximum principle}
\begin{proposition}\label{prop-omori-yau}
Let $g(t)$ with $t\in [0,T]$ be a smooth family of complete K\"ahler
metrics on $M$. Suppose that the sectional curvatures of $g(0)$ are
bounded and there is positive constant $C_0$ such that
\begin{equation*}
C_0^{-1}g(0)\leq g(t)\leq C_0g(0)
\end{equation*}
for any $t\in [0,T]$. Let $h\in C^\infty(M\times[0,T])$ be such that
\begin{equation*}
\sup_M h(0)<\sup_{M\times[0,T]}h<\infty.
\end{equation*}
Then, there is a sequence $(x_k,t_k)\in M\times (0,T]$ such that
\begin{equation*}
\lim_{k\to\infty}h(x_k,t_k)=\sup_{M\times[0,T]}h,\
\lim_{k\to\infty}\|\nabla h\|(x_k,t_k)=0,\ \pd{h}{t}(x_k,t_k)\geq 0,
\end{equation*}
and
\begin{equation*}
\limsup_{k\to\infty}\Delta h(x_k,t_k)\leq 0.
\end{equation*}
\end{proposition}
\begin{proof}
By adding a constant to $h$, we can suppose that
$\sup_{M\times[0,T]}h=1$. Let $p$ be a fixed point. Let $\rho$ be a
smooth function on $M$, such that
\begin{equation*}
\left\{\begin{array}{l}C_{1}^{-1}(1+r_0(p,x))\leq \rho(x)\leq
C_1(1+r_0(p,x))\\\|\nabla_0\rho\|\leq C_1\\ |\Delta_0\rho|\leq C_1
\end{array}\right.
\end{equation*}
all over $M$, where $C_1$ is some positive constant. (c.f. Theorem
3.6 in Shi \cite{Shi-RFU}.) Then
\begin{equation*}
|\Delta_t\rho|=|g^{\bar ji}(t)\rho_{i\bar j}|\leq C_0|g^{\bar
ji}(0)\rho_{i\bar j}|=C_0|\Delta_0\rho|\leq C_2.
\end{equation*}
all over $M\times [0,T]$ for some positive constant $C_2$.

Let $\eta$ be a smooth function on $[0,\infty)$ such that $\eta=1$
on $[0,1]$, $\eta=0$ on $[2,\infty)$, and $-2\leq \eta'\leq 0$.

For any $\epsilon \in (0,1/2)$, let $(x_0,t_0)\in M\times(0,T]$ be
such that
\begin{equation*}
\max\{1-\epsilon,\sup_Mh(0)\}<h(x_0,t_0)\leq 1.
\end{equation*}
Let $R>\rho(x_0)$ be a constant to be determined. Let
$\phi=\eta(\rho/R)$ and let $(\bar x,\bar t)$ be a maximum point of
$\phi h$. It is clear that $\bar t>0$ since
\begin{equation*}
h(\bar x,\bar t)\geq(\phi h)(\bar x,\bar t)\geq (\phi
h)(x_0,t_0)=h(x_0,t_0)>\sup_{M} h(x,0).
\end{equation*}
Moreover, we have
\begin{equation*}
1-\epsilon\leq (\phi h)(\bar x, \bar t)\leq 1, \nabla (\phi h)(\bar
x,\bar t)=0, \ \pd{(\phi h)}{t}(\bar x,\bar t)\geq 0
\end{equation*}
and
\begin{equation*}
\Delta (\phi h)(\bar x, \bar t)\leq 0.
\end{equation*}
By the first inequality and that $\sup_{M\times[0,s]}h=1$, we know
that
\begin{equation*}
\frac{1}{2}\leq 1-\epsilon\leq h(\bar x,\bar t)\leq 1, \ \mbox{and}\
\frac{1}{2}\leq (1-\epsilon)\leq \phi(\bar x,\bar t)\leq 1.
\end{equation*}
Then,
\begin{equation*}
\|\nabla h(\bar x,\bar t)\|=\frac{h\|\nabla \phi\|}{\phi}(\bar
x,\bar t)\leq \frac{C_3}{R},
\end{equation*}
\begin{equation*}
\pd{h}{t}(\bar x,\bar t)\geq 0,
\end{equation*}
and
\begin{equation*}
\Delta h(\bar x,\bar t)\leq-\frac{h\Delta \phi+2\vv<\nabla h,\nabla
\phi>}{\phi}(\bar x,\bar t)\leq \frac{C_4}{R}.
\end{equation*}
By choosing $R>\max\{\rho(x_0), \frac{C_3}{\epsilon},
\frac{C_4}{\epsilon}\}$, we get

\begin{equation*}
\|\nabla h(\bar x,\bar t)\|\leq \epsilon,\ \mbox{and}\ \Delta h(\bar
x,\bar t)\leq \epsilon.
\end{equation*}
Therefore, by choosing a sequence $\epsilon_k\to 0^+$, we get a
sequence $(x_k,t_k)$ satisfying our requirements.
\end{proof}

\section{curvatures estimates}
\begin{lemma}\label{lemma-local-boundness}
Let $g(t)$ be a solution to the K\"ahler-Ricci flow
(\ref{eqn-Kaehler-Ricci-flow}) on $[0,T]$ with $T<\infty$ satisfying
the following assumption:
\begin{equation}
\|\nabla ^kRm\|\leq C_k
\end{equation}
all over $M\times [0,T]$. Then, for any nonnegative integer $k$,
there is positive constant $A_k$, such that
\begin{equation*}
\|\nabla^k_{0}g\|_{g_0}\leq A_k
\end{equation*}
all over $M\times[0,T]$.
\begin{proof}
Because the curvatures are uniformly bounded, by the K\"ahler-Ricci
flow equation (\ref{eqn-Kaehler-Ricci-flow}), all the metrics are
uniformly equivalent. So, the lemma is true for $k=0$.

In the follows, ";" means taking covariant derivatives with respect
to $g_0$. By equation (\ref{eqn-Kaehler-Ricci-flow}), we have
\begin{equation}\label{eqn-g-i-j-k}
\begin{split}
\pd{g_{i\bar j;k}}{t}=&-R_{i\bar j;k}+cg_{i\bar
j;k}\\
=&-\big(\nabla_{k}R_{i\bar j}-(g_0)^{\bar\beta\alpha}R_{\alpha\bar
j}\nabla_{k}(g_0)_{i\bar\beta}\big)+cg_{i\bar j;k}\\
=&-\big[\nabla_{k}R_{i\bar j}+(g_0)^{\bar\beta\alpha}R_{\alpha\bar
j}\big(g^{\bar\delta\gamma}(g_0)_{\gamma\bar\beta}g_{i\bar\delta;k}\big)\big]+cg_{i\bar
j;k}.
\end{split}
\end{equation}
Moreover, by the assumption on curvatures and the case that $k=0$,
\begin{equation*}
\pd{\|\nabla_0g\|^2_{g_0}}{t}\leq C_1+C_2\|\nabla_0g\|^2_{g_0}
\end{equation*}
where  $C_1$ and  $C_2$ are some positive constants. Therefore, the
lemma is true for $k=1$ and $\|\nabla_0Rm\|$ is also uniformly
bounded since $\nabla_0Rm$ can be expressed as a combination of
$\nabla Rm$ and $\nabla_0g$.

Computing further on step by step by taking more covariant
derivatives with respect to $g_0$ on both sides of equation
(\ref{eqn-g-i-j-k}), we know that the lemma is true for all
nonnegative integer $k$.
\end{proof}
\end{lemma}
\begin{theorem}\label{thm-first-order}
Let $g(t)$ be a solution to the  K\"ahler-Ricci flow
(\ref{eqn-Kaehler-Ricci-flow}) on $[0,T)$ satisfying the following
assumption:
\begin{equation}
\|\nabla ^kRm\|\leq C(k,s)
\end{equation}
all over $M\times[0,s]$, for any nonnegative integer $k$ and any
$s\in (0,T)$, where  $T$ may be $\infty$ and $C(k,s)$ is a positive
constant depending on $k$ and $s$. Moreover, suppose that there is a
positive constant $C_0$ such that
\begin{equation*}
C_0^{-1}g_0\leq g(t)\leq C_0g_0
\end{equation*}
for any $t\in [0,T)$. Then, there is a positive constant $C$ such
that
\begin{equation*}
\|\nabla_0 g\|\leq C
\end{equation*}
all over $M\times [0,T)$.
\end{theorem}
\begin{proof} In the follows, covariant derivatives
are taken with respect to the initial metric $g_0$ and normal
coordinates are chosen with respect to $g_0$.

Let $S=(g_0)^{\bar j i}g_{i\bar j}$ and
\begin{equation*}
Q=g^{\bar qi}g^{\bar jp}g^{\bar r k}g_{i\bar j;k}g_{p\bar q;\bar r}.
\end{equation*}
We want to get an estimate of $Q$. By the assumptions, we have
$$nC_0^{-1}\leq S\leq nC_0.$$

By direct computation, we have
\begin{equation}
g_{i\bar j;k\bar l}=\big[(g_0)_{i\bar
\nu}\big((g_0)^{\bar\nu\mu}g_{\mu \bar j}\big)_k\big]_{\bar l}
=g_{i\bar j,k\bar l}+(g_0)_{i\bar \nu}\big((g_0)^{\bar
\nu\mu}\big)_{k\bar l}g_{\mu \bar j}=g_{i\bar j,k\bar l}+
(R_0)_{i\bar \mu k\bar l}g_{\mu\bar j},
\end{equation}
and
\begin{equation}\label{eqn-curvature-initial-metric}
\begin{split}
R_{i\bar jk\bar l}=&-g_{i\bar j,k\bar l}+g^{\bar
\nu\mu}g_{i\bar\nu;k}g_{\mu\bar j;\bar l}=-g_{i\bar j;k\bar
l}+g^{\bar \nu\mu}g_{i\bar\nu;k}g_{\mu\bar j;\bar l}+(R_0)_{i\bar
\mu k\bar l}g_{\mu\bar j},
\end{split}
\end{equation}
where a comma means a partial derivative. Hence
\begin{equation}\label{eqn-trace-of-second-derivative}
g^{\bar l k}g_{i\bar j;k\bar l}=-R_{i\bar j}+g^{\bar l k}g^{\bar
\nu\mu}g_{i\bar\nu;k}g_{\mu\bar j;\bar l}+ (R_0)_{i\bar \mu k\bar
l}g_{\mu\bar j}g^{\bar lk}.
\end{equation}
We are now ready to compute the evolution equation of $S$.
\begin{equation*}
\begin{split}
&\Big(\pd{}{t}-\Delta\Big)S\\
=&(g_0)^{\bar j i}(-R_{i\bar j}+cg_{i\bar j})-g^{\bar kl}S_{k\bar l}\\
=&-(g_0)^{\bar j i}(-R_{i\bar j}+cg_{i\bar j})-g^{\bar kl}(g_0)^{\bar ji}g_{i\bar j;k\bar l}\\
=&-(g_0)^{\bar j i}(-R_{i\bar j}+cg_{i\bar j})-(g_0)^{\bar
ji}\big(-R_{i\bar j}+g^{\bar l k}g^{\bar
\nu\mu}g_{i\bar\nu;k}g_{\mu\bar j;\bar l}+ (R_0)_{i\bar \mu k\bar
l}g_{\mu\bar j}g^{\bar lk}\big)\\
=&-(g_0)^{\bar ji}g^{l\bar k}g^{\bar \nu\mu}g_{i\bar\nu;k}g_{\mu\bar
j;\bar l}+cS-(R_0)_{i\bar j k\bar l}g_{j\bar i}g^{l\bar k}\\
\leq& -c_1 Q+C_2,
\end{split}
\end{equation*}
since the curvatures of $g_0$ is bounded and that the metrics are
uniformly equivalent.

A similar computations as in Appendix A of Yau
\cite{Yau-1978-Ricci-flat}(See Remark \ref{rm-evolution-Q}) give us
the following evolution equation of $Q$.
\begin{equation}\label{eqn-evolution-Q}
\begin{split}
&\Big(\pd{}{t}-\Delta\Big)Q\\
=&-\sum_{i,j,k,l}\frac{1}{\lambda_{i}\lambda_j\lambda_k\lambda_l}\Big|g_{i\bar
k;j\bar l}-\sum_{\gamma}\frac{1}{\lambda_\gamma}g_{i\bar\gamma
;j}g_{\gamma\bar k;\bar
l}\Big|^2\\
&-\sum_{i,q,k,\mu}\frac{1}{\lambda_i\lambda_k\lambda_q\lambda_\mu}\Big|g_{i\bar
q;k\mu}-\sum_{\alpha}\frac{1}{\lambda_\alpha}(g_{\alpha\bar q
;i}g_{k\bar\alpha;\mu}+g_{\alpha\bar q;\mu}g_{i\bar\alpha;
k})\Big|^2+\mathfrak{R}
\end{split}
\end{equation}
with
\begin{equation*}
|\mathfrak{R}|\leq C_3Q+C_4,
\end{equation*}
where we have chosen a normal coordinate of $g_0$ such that
$g_{i\bar j}=\lambda_i\delta_{ij}.$ Then, by choosing $C_5$ be such
that $c_1C_5-C_3=1$,
\begin{equation}\label{eqn-Q-S}
\Big(\pd{}{t}-\Delta\Big)(Q+C_5S)\leq-(Q+C_5S)+C_6.
\end{equation}

Fixed $s\in (0,T)$, by Lemma \ref{lemma-local-boundness}, $Q+C_5S$
is a bounded function on $M\times [0,s]$. By Proposition
\ref{prop-omori-yau} on any closed interval $[0,s]$, if
\begin{equation*}
\sup_{M\times [0,s]}(Q+C_5S)>\sup_{M}(Q+C_5S)(\cdot,0)=nC_5,
\end{equation*}
there exits a sequence $(x_k,t_k)\in M\times(0,s]$, such that
\begin{equation*}
\lim_{k\to\infty}(Q+C_5S)(x_k,t_k)=\sup_{M\times [0,s]}(Q+C_5S),\
\pd{(Q+C_5S)}{t}(x_k,t_k)\geq 0,
\end{equation*}
and
\begin{equation*}
\limsup_{k\to\infty}\Delta(Q+C_5S)(x_k,t_k)\leq 0.
\end{equation*}
Then, by (\ref{eqn-Q-S}),
\begin{equation*}
\begin{split}
0\leq&
\liminf_{k\to\infty}\Big(\pd{}{t}-\Delta\Big)(Q+C_5S)(x_k,t_k)\\
\leq&-\lim_{k\to\infty}(Q+C_5S)(x_k,t_k)+C_6\\
=&-\sup_{M\times [0,s]}(Q+C_5S)+C_6.
\end{split}
\end{equation*}
That is,
$$\sup_{M\times [0,s]}(Q+C_5S)\leq C_5.$$
Since $s\in (0,T)$ is arbitrary,
\begin{equation*}
Q\leq Q+C_5S\leq \max\{C_6,nC_5\}
\end{equation*}
all over $M\times [0,T)$.
\end{proof}
\begin{remark}\label{rm-evolution-Q}
We explain the computation of the evolution equation
(\ref{eqn-evolution-Q}) of $Q$ in more details. Note that, locally,
we have
\begin{equation*}
g_{i\bar j}=u_{i\bar j}+(g_0)_{i\bar j}.
\end{equation*}
Then, in a local coordinate, the K\"ahler-Ricci flow
(\ref{eqn-Kaehler-Ricci-flow}) becomes
\begin{equation*}
(u_t)_{i\bar j}=\Big(\log\frac{\det(u_{k\bar l}+(g_0)_{k\bar
l})}{\det(g_0)_{i\bar j}}\Big)_{i\bar j}+(R_0)_{i\bar
j}+c(g_0)_{i\bar j}
\end{equation*}
which makes our settings the same as in \S 8 of Chau
\cite{Chau-2004-JDG}. An detailed computation can be found in Shi \cite{Shi-RFU}.
\end{remark}

\begin{theorem}\label{thm-higher-order}
Let assumptions be the same as in the last thoerem. Then, for any
nonnegative integer $k$, there is a positive constant $A_k$, such
that $$\|\nabla^k_{0} g \|\leq A_k$$ all over $M\times [0,T)$.
\end{theorem}
\begin{proof} We prove it by induction on $k$. When $k=0$, it is by
assumptions. When $k=1$, it is just the last theorem. Suppose that
the inequality is true for $k=0,1,\cdots, m-1$, we want to get the
inequality for $k=m$. Because the metrics are uniformly equivalent,
it suffices to give an estimate to the quantity
\begin{equation*}
Q_m=\|\nabla_0^mg\|_{g_0}^{2}=(g_0)^{\beta_1\alpha_1}\cdots
(g_0)^{\beta_m\alpha_m}(g_0)^{\bar li}(g_0)^{\bar jk}g_{i\bar
j;\alpha_1\alpha_2\cdots\alpha_m}g_{k\bar l;
\beta_1\beta_2\cdots\beta_m},
\end{equation*}
where $\alpha_i's$ and $\beta_i's$ belong to $\{1,2,\cdots,n,\bar
1,\bar 2,\cdots,\bar n\}$.
\begin{equation*}
\begin{split}
&\Big(\pd{}{t}-\Delta\Big)Q_m\\
=&2(g_0)^{\beta_1\alpha_1}\cdots (g_0)^{\beta_m\alpha_m}(g_0)^{\bar
li}(g_0)^{\bar jk}(-R_{i\bar
j;\alpha_1\alpha_2\cdots\alpha_m}+cg_{i\bar
j;\alpha_1\alpha_2\cdots\alpha_m})g_{k\bar l;
\beta_1\beta_2\cdots\beta_m}\\
&-2g^{\bar \mu\lambda}(g_0)^{\beta_1\alpha_1}\cdots
(g_0)^{\beta_m\alpha_m}(g_0)^{\bar li}(g_0)^{\bar jk}g_{i\bar
j;\alpha_1\alpha_2\cdots\alpha_m\lambda\bar\mu}g_{k\bar l;
\beta_1\beta_2\cdots\beta_m}\\
&-2g^{\bar \mu\lambda}(g_0)^{\beta_1\alpha_1}\cdots
(g_0)^{\beta_m\alpha_m}(g_0)^{\bar li}(g_0)^{\bar jk}g_{i\bar
j;\alpha_1\alpha_2\cdots\alpha_m\lambda}g_{k\bar l;
\beta_1\beta_2\cdots\beta_m\bar\mu}\\
\leq& -2(g_0)^{\beta_1\alpha_1}\cdots
(g_0)^{\beta_m\alpha_m}(g_0)^{\bar li}(g_0)^{\bar jk}R_{i\bar
j;\alpha_1\alpha_2\cdots\alpha_m}g_{k\bar l;
\beta_1\beta_2\cdots\beta_m}\\
&-2g^{\bar\mu\lambda}(g_0)^{\beta_1\alpha_1}\cdots
(g_0)^{\beta_m\alpha_m}(g_0)^{\bar li}(g_0)^{\bar jk}g_{i\bar
j;\lambda\bar\mu\alpha_1\alpha_2\cdots\alpha_m}g_{k\bar l;
\beta_1\beta_2\cdots\beta_m}\\
&+C_1Q_m+C_2Q_m^{\frac{1}{2}}-\frac{1}{C_0}Q_{m+1}\\
\leq&-2(g_0)^{\beta_1\alpha_1}\cdots
(g_0)^{\beta_m\alpha_m}(g_0)^{\bar li}(g_0)^{\bar jk}R_{i\bar
j;\alpha_1\alpha_2\cdots\alpha_m}g_{k\bar l;
\beta_1\beta_2\cdots\beta_m}\\
&-2(g_0)^{\beta_1\alpha_1}\cdots (g_0)^{\beta_m\alpha_m}(g_0)^{\bar
li}(g_0)^{\bar jk}(g^{\bar\mu\lambda}g_{i\bar
j;\lambda\bar\mu})_{;\alpha_1\alpha_2\cdots\alpha_m}g_{k\bar l;
\beta_1\beta_2\cdots\beta_m}\\
&+C_6Q_m^{\frac{1}{2}}Q_{m+1}^\frac{1}{2}+C_5Q_m^{\frac{3}{2}}+C_4Q_m+C_3Q_m^{\frac{1}{2}}-\frac{1}{C_0}Q_{m+1}\\
\end{split}
\end{equation*}
where we have used the Ricci identity and the induction hypothesis.

Substituting identity (\ref{eqn-trace-of-second-derivative}) into the last inequality, we get
\begin{equation*}
\begin{split}
&\Big(\pd{}{t}-\Delta\Big)Q_m\\
\leq&-2(g_0)^{\beta_1\alpha_1}\cdots (g_0)^{\beta_m\alpha_m}(g_0)^{\bar
li}(g_0)^{\bar jk}R_{i\bar j;\alpha_1\alpha_2\cdots\alpha_m}g_{k\bar
l;
\beta_1\beta_2\cdots\beta_m}\\
&-2(g_0)^{\beta_1\alpha_1}\cdots (g_0)^{\beta_m\alpha_m}(g_0)^{\bar
li}(g_0)^{\bar jk}(-R_{i\bar j}+g^{\bar \delta\gamma}g^{\bar
\nu\mu}g_{i\bar\nu;\gamma}g_{\mu\bar j;\bar \delta}\\
&+ (R_0)_{i\bar \mu \gamma\bar \delta}g_{\mu\bar j}g^{\bar
lk})_{;\alpha_1\alpha_2\cdots\alpha_m}g_{k\bar l;
\beta_1\beta_2\cdots\beta_m}\\
&+C_6Q_m^{\frac{1}{2}}Q_{m+1}^\frac{1}{2}+C_5Q_m^{\frac{3}{2}}+C_4Q_m+C_3Q_m^{\frac{1}{2}}-\frac{1}{C_0}Q_{m+1}\\
\leq&
C_{10}Q_m^{\frac{1}{2}}Q_{m+1}^\frac{1}{2}+C_9Q_m^{\frac{3}{2}}+C_8Q_m+C_7Q_m^{\frac{1}{2}}-\frac{1}{C_0}Q_{m+1}\\
\leq &C_{11}Q_{m}^\frac{3}{2}+C_{12}-c_{13}Q_{m+1}.
\end{split}
\end{equation*}

The same computation using the induction hypothesis provides us
\begin{equation*}
\Big(\pd{}{t}-\Delta \Big)Q_{m-1}\leq C_{14}-c_{15}Q_{m}.
\end{equation*}
Moreover, note that
\begin{equation*}
\begin{split}
&|\vv<\nabla Q_{m-1},\nabla Q_{m-1}> |\\ =&|g^{\bar\mu\lambda}(Q_{m-1})_{\lambda}(Q_m)_{\bar\mu}|\\
=&|g^{\bar\mu\lambda}(g_0)^{\beta_1\alpha_1}\cdots
(g_0)^{\beta_{m-1}\alpha_{m-1}}(g_0)^{\bar li}(g_0)^{\bar
jk}(g_{i\bar j;\alpha_1\alpha_2\cdots\alpha_{m-1}}g_{k\bar l;
\beta_1\beta_2\cdots\beta_{m-1}})_{;\lambda}\times\\
&(g_0)^{\beta'_1\alpha'_1}\cdots
(g_0)^{\beta'_m\alpha'_m}(g_0)^{\bar l'i'}(g_0)^{\bar j'k'}(g_{i\bar
j;\alpha'_1\alpha'_2\cdots\alpha'_m}g_{k'\bar l';
\beta'_1\beta'_2\cdots\beta'_m})_{\bar\mu}|\\
\leq&C_{16}Q_{m}Q_{m+1}^\frac{1}{2}
\end{split}
\end{equation*}
where we have used that the metrics are uniformly equivalent and the
induction hypothesis.

Let $Q=(Q_{m-1}+A)Q_m$ with $A$ a positive constant to be
determined. We have,
\begin{equation*}
\begin{split}
&\Big(\pd{}{t}-\Delta\Big)Q\\
=&Q_m
\Big(\pd{}{t}-\Delta\Big)Q_{m-1}+(Q_{m-1}+A)\Big(\pd{}{t}-\Delta\Big)Q_m
-2\Ree\{\vv<\nabla
Q_{m-1},\nabla Q_m >\}\\
\leq&Q_m (-c_{15}Q_m +C_{14})+(Q_{m-1}+A)(-c_{13}Q_{m+1}+C_{11}Q_m ^{\frac{3}{2}}+C_{12})+2C_{16}Q_m Q_{m+1}^\frac{1}{2}\\
\leq&-\frac{c_{15}}{2}Q_m ^{2}+C_{14}
Q_m+(Q_{m-1}+A)(-c_{13}Q_{m+1}+C_{11}Q_m ^{\frac{3}{2}}+C_{12})+\frac{2C_{16}^{2}}{c_{15}}Q_{m+1}\\
\leq&
-\frac{c_{15}}{2}Q_m ^{2}+C_{14}Q_m +(Q_{m-1}+A)(C_{11}Q_m ^{\frac{3}{2}}+C_{12})\\
\leq& -c_{16}Q^{2}+C_{17}
\end{split}
\end{equation*}
if we chose $A$ be such that $-Ac_{13}+\frac{2C_{16}^2}{c_{15}}=0$,
where we have used the induction hypothesis.

Similar as in the proof of Theorem \ref{thm-first-order} using
Proposition \ref{prop-omori-yau}, we have
\begin{equation*}
Q\leq \sqrt{\frac{C_{17}}{c_{16}}}.
\end{equation*}
Therefore, $Q_m\leq A^{-1}\sqrt{\frac{C_{17}}{c_{16}}}$. This
completes the proof.
\end{proof}
\begin{remark}
The trick used in the proof is basically due to Chau
\cite{Chau-2004-JDG}  (\S 9). We only simplify it by using norms
with respect to the initial metric $g_0$ instead of using $g$ in
Chau \cite{Chau-2004-JDG}.
\end{remark}
\begin{corollary}\label{cor-curvature-bound-1}
Let assumptions be the same as in Theorem \ref{thm-first-order}.
Then, for any nonnegative integer $k$, there is a positive constant
$A_k$ such that $$\|\nabla^kRm\|\leq A_k$$ all over $M\times[0,T)$.
\end{corollary}
\begin{proof} Note the by equation
(\ref{eqn-curvature-initial-metric}), we have
\begin{equation*}
R_{i\bar jk\bar l}=-g_{i\bar j;k\bar l}+g^{\bar
\nu\mu}g_{i\bar\nu;k}g_{\mu\bar j;\bar l}+(R_0)_{i\bar \mu k\bar
l}g_{\mu\bar j},
\end{equation*}
where covariant derivatives ";" and normal coordinates are taken
with respect to $g_0$. Hence, by the last theorem, for any
nonnegative integer $k$, there is a positive constant $C_k$ such
that
\begin{equation*}
\|\nabla_0^{k}Rm\|\leq C_k
\end{equation*}
all over $M\times [0,T)$. In particular, the corollary is true for
$k=0$.

When $k=1$, note that
\begin{equation*}
\begin{split}
\nabla_{\lambda}R_{i\bar jk\bar
l}=&(g^{\bar\beta\alpha}g^{\bar\delta\gamma}R_{\alpha\bar j
\gamma\bar l})_{,\lambda}g_{i\bar\beta}g_{k\bar\delta}\\
=&R_{i\bar jk\bar l;\lambda}-g^{\bar\beta\alpha}R_{\alpha\bar j
k\bar l}g_{i\bar\beta;\lambda}-g^{\bar\delta\gamma}R_{i\bar j
\gamma\bar l}g_{k\bar\delta;\lambda}.
\end{split}
\end{equation*}
Hence the corollary is true for $k=1$.

Computing step by step, we can express $\nabla^k Rm$ as a
combination of covariant derivatives of $g$ and $Rm_0$. Therefore,
the corollary is true for all nonnegative integer $k$.

\end{proof}
We come to prove the main result of this article.

\begin{theorem}\label{thm-curvature-bound}Let $g(t)$ with $t\in [0,T)$ be a complete solution to the
K\"ahler-Ricci flow (\ref{eqn-Kaehler-Ricci-flow}) satisfying Shi's
estimate (\ref{eqn-Shi-estimate}). Suppose that there a is positive
constant $C$ such that
\begin{equation*}
C^{-1}g(0)\leq g(t)\leq Cg(0).
\end{equation*}
Then, for any nonnegative integer $k$, there are two positive
constants $A_k$ and $B_k$ such that
\begin{equation*}
\|\nabla^kRm\|^2(x,t)\leq A_k+\frac{B_k}{t^k}
\end{equation*}
for any $(x,t)\in M\times [0,T)$.
\end{theorem}
\begin{proof} Fixed $\epsilon\in (0,T)$. Then, the family $g(t)$ with $t\in [\epsilon, T)$ is a
solution the K\"ahler-Ricci flow (\ref{eqn-Kaehler-Ricci-flow}) satisfies the assumption of Corollary \ref{cor-curvature-bound-1}.
Hence,
\begin{equation*}
\|\nabla^kRm\|^2(x,t)\leq A_k
\end{equation*}
for any $(x,t)\in M\times[\epsilon,T)$. Combining this with Shi's estimate, we complete the proof.

\end{proof}
\begin{remark}
Theorem \ref{thm-higher-order} is sometimes more useful than Theorem \ref{thm-curvature-bound}. For example,
when $T=\infty$, we can extract convergent subsequence of metrics by Theorem \ref{thm-higher-order}.
\end{remark}
\begin{remark}
The method will not work for the real case. The first obstacle is that we do not have an identity
which is similar with (\ref{eqn-trace-of-second-derivative}).

\end{remark}
\section{An application.}
\begin{theorem}
Let $\tilde g$ be a complete K\"ahler metric on $M^n$ with bounded sectional curvature. Suppose that
\begin{equation*}
-\tilde R_{i\bar j}+c\tilde g_{i\bar j}=f_{i\bar j}
\end{equation*}
with $f$ a smooth bounded  function on $M^n$. Then, the K\"ahler-Ricci flow (\ref{eqn-Kaehler-Ricci-flow})
with initial metric $\tilde g$ has a long time solution satisfying Shi's estimate (\ref{eqn-Shi-estimate}).
\end{theorem}
\begin{proof}Let $g(t)$ with $[0,T)$ be maximal complete solution to the K\"ahler-Ricci flow (\ref{eqn-Kaehler-Ricci-flow})
with initial data $g(0)=\tilde g$ satisfying Shi's estimate (\ref{eqn-Kaehler-Ricci-flow}). It suffices to
show that $T=\infty$. Suppose that $T<\infty$, we want to get a contradiction.

As in Chau \cite{Chau-2004-JDG}, the K\"ahler-Ricci flow (\ref{eqn-Kaehler-Ricci-flow}) is related to
to the parabolic Monge-Amp\`ere equation:
\begin{equation}\label{eqn-monge-ampere}
\left\{\begin{array}{l}\pd{u}{t}=\log\frac{\det(\tilde g_{i\bar j}+u_{i\bar j})}{\det \tilde g_{i\bar j}}+cu+f\\u(x,0)=0.
\end{array}\right.
\end{equation}
$g_{i\bar j}=\tilde g_{i\bar j}+u_{i\bar j}$ is a solution of the K\"ahler-Ricci flow (\ref{eqn-Kaehler-Ricci-flow})
if $u$ is a solution to equation (\ref{eqn-monge-ampere}).

Let $w=u_t$. Then,
\begin{equation*}
\left\{\begin{array}{l}\pd{w}{t}=\Delta w+cw\\w(x,0)=f(x).
\end{array}\right.
\end{equation*}
By maximum principle, we know that
\begin{equation*}
|w|(x,t)\leq e^{ct}\sup |f|\leq e^{cT}\sup |f|:=C_1.
\end{equation*}
for any $(x,t)\in M\times [0,T)$. Moreover
\begin{equation*}
|u|(x,t)\leq C_1T:=C_2
\end{equation*}
for any $(x,t)\in M\times[0,T)$. By the same argument as in second order estimate of Monge-Amp\`ere equation
as in Yau \cite{Yau-1978-Ricci-flat} (see also Chau \cite{Chau-2004-JDG}), we know that
\begin{equation*}
\tilde g^{\bar ji}g_{i\bar j}=n+\Delta u\leq C_3.
\end{equation*}
By equation (\ref{eqn-monge-ampere}),
\begin{equation*}
\frac{\det g_{i\bar j}}{\det{\tilde g_{i\bar j}}}\geq c_4.
\end{equation*}
Let $\lambda_1\leq\lambda_2\leq\cdots\leq \lambda_n$ be the $n$ eigenvalues of $g_{i\bar j}$ with respect to $\tilde g_{i\bar j}$.
Then, we have
\begin{equation*}
\lambda_1+\lambda_2+\cdots+\lambda_n\leq C_3\ \mbox{and}\ \lambda_1\cdot\lambda_2\cdots\lambda_n\geq c_4.
\end{equation*}
Therefore $\lambda_n\leq C_3$ and $\lambda_1\geq \frac{c_4}{\lambda_2\lambda_3\cdots\lambda_n}\geq c_4C_3^{-n+1}:=c_5$ which
implies that $g(t)$ is uniformly equivalent to $\tilde g$. By Theorem \ref{thm-curvature-bound}, curvatures of
$g(t)$ is uniformly bounded. This violates that $g(t)$ is a maximal solution.
\end{proof}
\begin{remark}
When $M^n$ is a compact K\"ahler manifold, if the potential $f$ exists for the initial metric, then $f$
is automatically bounded. Hence, the K\"ahler-Ricci flow (\ref{eqn-Kaehler-Ricci-flow}) has a long time solution
in this case. This is just the long time existence result in Cao \cite{Cao}.
\end{remark}


\begin{thebibliography}{99}
\bibitem{Cao}
Cao, Huai Dong {\sl Deformation of K\"ahler metrics to K\"ahler-Einstein metrics on compact K\"ahler manifolds.}  Invent. Math.  81  (1985),  no. 2, 359--372.

\bibitem{Chau-2004-JDG}
 Chau, A., {\it Convergence of the K\"ahler-Ricci flow on noncompact K\"ahler manifolds},
 J. Differential Geom. 66 (2004), no. 2, 211--232.

\bibitem{Chen-Zhu-uniqueness}
 Chen, Bing-Long; Zhu, Xi-Ping.{\sl
 Uniqueness of the Ricci flow on complete noncompact manifolds.}
  J. Differential Geom. 74 (2006), no. 1, 119--154.

\bibitem{Hamilton-Formation of singularity}
Hamilton, R., {\it The formation of singularities in the Ricci
flow}, Surv. Differ. Geom. 2 (1995), 7-136, International Press.

\bibitem{Perelman1}
Perelman, G., {\it The entropy formula for Ricci flow and its
geometric applications}, arXiv://math.DG/0211159.

\bibitem{sesum}
\v Se\v sum,  Nata\v sa, {\sl Curvature tensor under the Ricci
flow.} Amer. J. Math. 127 (2005), no. 6, 1315--1324.

\bibitem{Shi-RFU}
Shi, W.-X.,  {\it Ricci flow and the uniformization on complete
noncompact K\"{a}hler manifolds}, J. Differential Geom. 45 (1997),
94-220.
\bibitem{Shi-Deforming metrics}
Shi, W.-X., {\it Deforming the metric on complete Riemannian
manifolds}, J. Differential Geom. 30 (1989), 223-301.

\bibitem{Yau-Schwartz-lemma}
Yau, S.-T., {\it A general Schwarz lemma for K\"ahler manifolds},
Amer. J. Math. 100 (1978), no. 1, 197--203.

\bibitem{Yau-1978-Ricci-flat}
Yau, S.-T., {\it On the Ricci curvature of a compact K\"ahler
manifold and the complex Monge-Amp¨¨re equation. I}, Comm. Pure
Appl. Math. 31 (1978), no. 3, 339--411.

\end{thebibliography}
\end{document}